\documentclass[10pt]{amsart}
\usepackage{amsmath}
\usepackage{graphicx} 
\usepackage{listings} 
\usepackage{hyperref,xcolor}
\usepackage[T1]{fontenc}
\usepackage{pdfsync}
\usepackage{yfonts}      
\usepackage{tcolorbox}   
\usepackage{lipsum} 

\DeclareMathOperator{\QQ}{\mathcal Q}
\DeclareMathOperator{\PP}{\mathcal P}
\DeclareMathOperator{\KK}{\mathcal K}
\DeclareMathOperator{\LL}{\mathcal L}

\newtheorem{theorem}{Theorem}
\newtheorem{conjecture}[theorem]{Conjecture}

\newtheorem{definition}[theorem]{Definition}

\newtheorem{example}[theorem]{Example}

\DeclareMathSymbol{\lsim}{\mathord}{symbols}{"18}
\title{Periodicity conjectures\\ for all $2$-sumfree sequences}
\author[van Berkel\and Bosma]{Daan van Berkel\and Wieb Bosma}
\email{daan.v.berkel.1980@gmail.com\and bosma@math.ru.nl} 
\date{today}

\begin{document}
\begin{abstract}
Complementing an earlier paper \cite{BeBo}, which focused on $3$-sumfree
sequences, we here consider only $2$-sumfree sequences:
starting with positive integers $f$ and $g>f$, the infinite,
increasing 2-sumfree sequence $S_{f,g}$ is constructed as follows. 
After any initial segment,
the next entry is the smallest positive integer exceeding
all previous ones that differs from all sums
of distinct pairs in the sequence. It follows from a theorem in
\cite{BeBo} that for every $f\geq 1$ and all $f+1\leq g<2f$ the sequence
$S_{f,g}$ exhibits ultimately periodic behaviour. In this paper
we state precise conjectures that, if true, would imply that
{\it every} $2$-sumfree sequence is ultimately periodic. Here periodicity
of an increasing sequence
is understood to mean periodicity of the sequence of first differences,
or, equivalently, of its characteristic sequence.
We supply much computational evidence to support the conjectures.
\end{abstract}

\subjclass[2000]{}
\maketitle

\section{Introduction}
The strict, greedy, $2$-sumfree sequences $S_{f,g}$, which are the main
object of study in this paper, are infinite sequences of increasing
positive integers
obtained as follows: initialize the first two values $s_0=f\geq 1$ and 
$s_1=g>f$, and once the first $k\geq 2$ values $s_0, s_1, \cdots, s_{k-1}$
 are defined, let $s_k$
be the smallest integer greater than $s_{k-1}$ that is not equal to the sum
of 2 different previous elements in the sequence.
This 2-sumfree sequence construction is {\it strict} in the sense that we
require that $s_k$ is not equal to the sum of two {\it distinct} smaller
entries, and {\it greedy} in the sense that $s_k$ is the smallest possible
value.
The most interesting problem concerning them is the
{\it periodicity question}, which we describe next.

It turns out that most 2-sumfree sequences exhibit a regular pattern
after a while, namely that the related sequence of first differences
$D_{f,g}=(s_i-s_{i-1})_{i\geq 1}$ for $S_{f,g}=(s_i)_{i\geq 0}$
is eventually periodic.

Consider the case $f=3, g=7$, for example. 
It is not hard to compute the first 26 terms of $S_{3,7}$:
$$3, 7, 8, 9, 13, 14, 18, 19, 24, 29, 30, 34, 35, 40, 45, 46, 50, 51, 56, 61, 
62, 66, 67, 72, 77,78$$
which gives the following initial segment of the difference sequence $D_{3,7}$:
$$4, 1, 1, 4, 1, 4, 1, 5, 5, 1, 4, 1, 5, 5, 1, 4, 1, 5, 5, 1, 4, 1, 5, 5, 1.$$
One easily recognizes the pattern: after the first 4 terms,
the period $(1,4,1,5,5)$ repeats. So it seems that after the first 4
terms of $S_{3,7}$ a pattern of length 5 is established. By slight abuse
of terminology we will call $S_{f,g}$ {\it eventually periodic}
with pre-period of length $k$ and period of length $p$ if the
associated sequence $D_{f,g}$ has preperiod length $k$ and period length
$p$. (We will always assume that these are taken to be minimal.)

Alternatively, one may look at the associated {\it characteristic sequence}
$C_{f, g}$ of $S_{f,g}$: it is the infinite sequence of zeroes and
ones defined by $C_{f, g}=(c_i)_{i\geq 1}$ with $c_i=1$ if and only if
$i\in S_{f,g}$.

Continuing the above example we easily find for $C_{3,7}$ the initial segment
$$0^2,1,0^3,1^3,0^3,1^2,0^3,1^2,0^4,1,0^4,1^2,0^3,1^2,0^4,1,0^4,1,0^3,1^2,0^4,1,0^4,1^2,0^3,$$ 
where $0^2$ denotes $0,0$, etcetera. The sequence $C_{3,7}$ can be
seen to be periodic after a preperiod (of length 9) containing
4 ones (corresponding to the first four entries 3, 7, 8, 9 of $S_{3,7}$)
and a period (of length 16) containing 5 ones:
$$C_{3,7}=(0^2,1,0^3,1^3,\overline{0^3,1^2,0^3,1^2,0^4,1,0}),$$
where the the bar indicates the repeated period.
It is not
difficult to see that $D_{f,g}$ is ultimately periodic with
preperiod length $k$ and period length $p$ if and only if $C_{f,g}$
is ultimately periodic with a preperiod containing $k$ ones and a period
containing $p$ ones. More conveniently, the ultimate
periodicity of $C_{f,g}$ also implies that $S_{f,g}$ itself becomes
properly ultimately periodic when considered modulo $m$, where $m$
is the length of the period of $C_{f,g}$, as well as the sum of the
entries in the period of $D_{f, g}$. Indeed, in our example we
see that
$$S_{3,7}\equiv (3, 7, 8, 9, \overline{13, 14, 2, 3, 8})\bmod 16,$$
(We will not make
make much use of this `modular periodicity', but it is important
that $S_{f,g}$ is to be recovered from such modular period in a greedy
fashion: for each index the {\it smallest possible} integer in the given
residue class is chosen that makes the sequence increasing.)

Informal statements about periodicity of $S_{f,g}$
should henceforth be interpreted
as proper statements about periodicity of $D_{f,g}$ or, equivalently,
of $C_{f,g}$.

The question we would like to address is the main conjecture
in this area.
\begin{conjecture}\label{conj:main}
Every 2-sumfree sequence is eventually periodic.
\end{conjecture}
We have not been able to establish the truth of this conjecture,
but instead offer 3 pieces of evidence in support.
The first consists of theorems for families of special cases in which it holds,
in particular for $g\leq 2f$.
The second is extensive computational evidence to show that it
is true in `small' instances; note that this means that we have
computed enough terms of these cases to be able to {\it prove}
ultimate periodicity. These proofs make use of a lemma, stating
that sufficiently many repetitions of a non-trivial subsequence
imply periodicity (see Corollary \ref{corollary}).
Thirdly, and taking most space in this paper, based on these
computations and the patterns observed in them,
we are able to offer a precise conjecture for
the length of the preperiod and of the period of $S_{f,g}$
for any pair $(f, g)$.
For the statements of these refined conjectures, see 
Sections \ref{sec:period} and \ref{sec:preperiod}. 

\begin{example}\rm
To give an idea of the challenges faced, and to establish some notation,
we describe some results for small cases here. 

The two $20\times 20$ matrices $K$ and $P$
below contain values for all lengths of the
preperiod $k$ and of the period $p$ of any 2-sumfree sequence $S_{f,g}$
with $1\leq f\leq 20$ and $g$ such that $d=g-f\leq 20$. The columns
correspond to a fixed value of $f$, and the rows to a fixed value
of $d$. (It is more convenient to use $d$ than $g$ as a parameter,
for instance because values for $g\leq f$ are not defined.)
So the values $k=5$ and $p=2$, respectively, at positions 
$(d,f)=(6,2)$ of
the two matrices indicate that $S_{2, 8}$ will be periodic after a
preperiod of length 5 with period 2. Indeed, $S_{2, 8}$ starts out as
follows:
$$S_{2, 8}=(2, 8, 9, 12, 13, 16, 19, 23, 26, 30, 33, 37, 40, 44, \cdots),$$
so the sequence of differences as
$$D_{2, 8}=(6, 1, 3, 1, 3, 3, 4, 3, 4, 3, 4, 3, 4, \cdots).$$
After the sixth entry $16$, the next values of $S_{2, 8}$ are obtained by
adding alternately 3 and 4, a period of length 2. Also, the sequence is
eventually periodic when taken modulo $7$.

{\tiny
$$
\begin{pmatrix}
2 & 3 & 4 & 5 & 6 & 7 & 8 & 9 & 10 & 11 & 12 & 13 & 14 & 15 & 16 & 17 & 18 & 19 & 20 & 21 \\ 
0 & 2 & 4 & 5 & 6 & 7 & 8 & 9 & 10 & 11 & 12 & 13 & 14 & 15 & 16 & 17 & 18 & 19 & 20 & 21 \\ 
4 & 5 & 4 & 5 & 6 & 7 & 8 & 9 & 10 & 11 & 12 & 13 & 14 & 15 & 16 & 17 & 18 & 19 & 20 & 21 \\ 
1 & 6 & 4 & 5 & 6 & 7 & 8 & 9 & 10 & 11 & 12 & 13 & 14 & 15 & 16 & 17 & 18 & 19 & 20 & 21 \\ 
5 & 4 & 9 & 5 & 6 & 7 & 8 & 9 & 10 & 11 & 12 & 13 & 14 & 15 & 16 & 17 & 18 & 19 & 20 & 21 \\ 
1 & 5 & 5 & 5 & 6 & 7 & 8 & 9 & 10 & 11 & 12 & 13 & 14 & 15 & 16 & 17 & 18 & 19 & 20 & 21 \\ 
7 & 7 & 5 & 11 & 6 & 7 & 8 & 9 & 10 & 11 & 12 & 13 & 14 & 15 & 16 & 17 & 18 & 19 & 20 & 21 \\ 
1 & 8 & 5 & 6 & 30 & 7 & 8 & 9 & 10 & 11 & 12 & 13 & 14 & 15 & 16 & 17 & 18 & 19 & 20 & 21 \\ 
8 & 6 & 19 & 6 & 13 & 21 & 8 & 9 & 10 & 11 & 12 & 13 & 14 & 15 & 16 & 17 & 18 & 19 & 20 & 21 \\ 
1 & 7 & 23 & 6 & 7 & 25 & 23 & 9 & 10 & 11 & 12 & 13 & 14 & 15 & 16 & 17 & 18 & 19 & 20 & 21 \\ 
9 & 9 & 26 & 6 & 7 & 15 & 39 & 25 & 10 & 11 & 12 & 13 & 14 & 15 & 16 & 17 & 18 & 19 & 20 & 21 \\ 
1 & 10 & 8 & 12 & 7 & 8 & 29 & 26 & 27 & 11 & 12 & 13 & 14 & 15 & 16 & 17 & 18 & 19 & 20 & 21 \\ 
10 & 26 & 8 & 13 & 7 & 8 & 17 & 31 & 28 & 29 & 12 & 13 & 14 & 15 & 16 & 17 & 18 & 19 & 20 & 21 \\ 
1 & 26 & 25 & 14 & 7 & 8 & 9 & 33 & 48 & 30 & 31 & 13 & 14 & 15 & 16 & 17 & 18 & 19 & 20 & 21 \\ 
11 & 11 & 36 & 13 & 14 & 8 & 9 & 19 & 35 & 31 & 32 & 33 & 14 & 15 & 16 & 17 & 18 & 19 & 20 & 21 \\ 
1 & 12 & 32 & 10 & 15 & 8 & 9 & 10 & 37 & 37 & 33 & 34 & 35 & 15 & 16 & 17 & 18 & 19 & 20 & 21 \\ 
12 & 31 & 69 & 10 & 16 & 8 & 9 & 10 & 21 & 39 & 57 & 35 & 36 & 37 & 16 & 17 & 18 & 19 & 20 & 21 \\ 
1 & 31 & 11 & 31 & 71 & 16 & 9 & 10 & 11 & 41 & 41 & 36 & 37 & 38 & 39 & 17 & 18 & 19 & 20 & 21 \\ 
13 & 13 & 34 & 43 & 16 & 17 & 9 & 10 & 11 & 23 & 43 & 43 & 38 & 39 & 40 & 41 & 18 & 19 & 20 & 21 \\ 
1 & 14 & 31 & 116 & 12 & 18 & 9 & 10 & 11 & 12 & 45 & 45 & 66 & 40 & 41 & 42 & 43 & 19 & 20 & 21  
\end{pmatrix}
$$
\centerline{The preperiod lengths $K(d,f)$ of $S_{f, f+d}$ for $1\leq f, d\leq 20$.}

}

{\tiny
$$
\begin{pmatrix}
1& 2& 3& 4& 5& 6& 7& 8& 9& 10& 11& 12& 13& 14& 15& 16& 17& 18& 19& 20 \\
1& 1& 3& 4& 5& 6& 7& 8& 9& 10& 11& 12& 13& 14& 15& 16& 17& 18& 19& 20 \\
4& 5& 4& 4& 5& 6& 7& 8& 9& 10& 11& 12& 13& 14& 15& 16& 17& 18& 19& 20 \\
1& 7& 5& 5& 5& 6& 7& 8& 9& 10& 11& 12& 13& 14& 15& 16& 17& 18& 19& 20 \\
3& 6& 2& 2& 6& 6& 7& 8& 9& 10& 11& 12& 13& 14& 15& 16& 17& 18& 19& 20 \\
1& 2& 4& 7& 7& 7& 7& 8& 9& 10& 11& 12& 13& 14& 15& 16& 17& 18& 19& 20 \\
11& 9& 2& 10& 8& 8& 8& 8& 9& 10& 11& 12& 13& 14& 15& 16& 17& 18& 19& 20 \\
1& 11& 4& 10& 3& 9& 9& 9& 9& 10& 11& 12& 13& 14& 15& 16& 17& 18& 19& 20 \\
13& 9& 4& 10& 12& 10& 10& 10& 10& 10& 11& 12& 13& 14& 15& 16& 17& 18& 19& 20 \\
1& 9& 5& 10& 12& 14& 11& 11& 11& 11& 11& 12& 13& 14& 15& 16& 17& 18& 19& 20 \\
15& 13& 2& 10& 12& 14& 4& 12& 12& 12& 12& 12& 13& 14& 15& 16& 17& 18& 19& 20 \\
1& 15& 13& 10& 12& 14& 16& 13& 13& 13& 13& 13& 13& 14& 15& 16& 17& 18& 19& 20 \\
17& 12& 12& 12& 12& 14& 16& 18& 14& 14& 14& 14& 14& 14& 15& 16& 17& 18& 19& 20 \\
1& 12& 11& 14& 12& 14& 16& 18& 5& 15& 15& 15& 15& 15& 15& 16& 17& 18& 19& 20 \\
19& 17& 26& 15& 12& 14& 16& 18& 20& 16& 16& 16& 16& 16& 16& 16& 17& 18& 19& 20 \\
1& 19& 21& 16& 14& 14& 16& 18& 20& 22& 17& 17& 17& 17& 17& 17& 17& 18& 19& 20 \\
21& 15& 2& 15& 16& 14& 16& 18& 20& 22& 6& 18& 18& 18& 18& 18& 18& 18& 19& 20 \\
1& 15& 18& 14& 3& 14& 16& 18& 20& 22& 24& 19& 19& 19& 19& 19& 19& 19& 19& 20 \\
23& 21& 16& 13& 18& 16& 16& 18& 20& 22& 24& 26& 20& 20& 20& 20& 20& 20& 20& 20 \\
1& 23& 14& 15& 19& 18& 16& 18& 20& 22& 24& 26& 7& 21& 21& 21& 21& 21& 21& 21 
\end{pmatrix}
$$
\centerline{The period lengths $P(d,f)$ of $S_{f, f+d}$, for $1\leq f, d\leq 20$.}

}

What strikes an observer immediately amidst some
general irregularity, is the regularity in `half'
of these matrices, namely, the part above the main diagonal.
Indeed this behaviour is proven in Theorem \ref{thm:upper}.

The behaviour we try to describe in this paper, is the regularity
of the growth of the entries in both matrices. In particular, we
intend to describe precisely the growth of the entries in the
{\it columns} of these matrices, that is, of functions ${\mathcal Q}_f$
and ${\mathcal P}_f$ that determine the (pre)period lengths
for fixed $f$ as a function of $d$. This regularity is much less obvious
from the small samples shown in the matrices above, but the observation is
the basis for the conjectures. What obscures the pattern at first
are three phenomena: the regularity in the column for $f$ below the diagonal
in fact consists of $2f$ regular sequences, one for each residue class
modulo $2f$; secondly, 
for small values of $f$ (in particular $f\leq 4$)
the regularity in the columns is slightly different; 
and finally, the initial part of the subsequences is often also irregular 
for larger $f$ (usually the first 2 or 3 values).

These exceptions for small values of parameters
makes the precise statements (that hold for {\it every $d, f\geq 1$})
somewhat involved (see Definitions \ref{def:LL} and \ref{def:KK}), 
but the rule of thumb is that growth in these
residue classes for $d\bmod 2f$ is usually linear, and sometimes quadratic,
and we give precise values for these linear and quadratic functions in
every case.

The computational results referred to above and described
in Section \ref{sec:computations}
consisted of the computation of the $250\times 250$
 preperiod-matrix and the $500\times 500$ matrix of period lengths.
Section \ref{sec:period} contains the functions $\LL_f$
that are conjectured to describe essentially the infinite
columns of the period matrix as a function of row number $d$.
Section \ref{sec:preperiod}
describes the (more complicated) growth of the column 
for $f$ of the preperiod length as a function of $d$.
\end{example}
To conclude this Introduction: it is our firm belief that
the stated conjectures correctly predict all (pre)period
lengths. It seems likely that, with effort, every individual
sequence $S_{f,g}$ could be analyzed (as is done for 
a few examples in the following sections), but a general proof
may be difficult to formulate. What could perhaps be achieved
is a very general argument that ultimate periodicity is inevitable.

\section{The conjectures for period lengths}\label{sec:period}
In this section we will conjecture precisely how,
for fixed $f$, the period length $\PP(d,f)$ of the sequence $S_{f, f+d}$
depends on $d$.
\begin{definition}\label{def:LL}
For $f\ge 1$, define $q_{-}=\lfloor\frac{f}{2}\rfloor$ and $q^{+}=\lceil\frac{f}{2}\rceil$;
note that $q_{-}+q^{+}=f$ for every $f$ and $q_{-}=q^{+}$ for even $f$.

Then the function  $\LL_{f}$ on
$d\geq 1$ is defined as follows: 
\begin{itemize}
\item[(1)] $\LL_f(d)=f$ if $1\leq d\leq f-1$;
\end{itemize}
and for $d\geq f$, blocks of length $2f$ will be defined,
for every $k\geq 1$ by:
\begin{itemize}
\item[(2)] $\LL_f(d)=k(f+1+j)$, if $d=(2k-1)f+j$ for $0\leq j\leq q_{-}$;
\item[(3)] if $f$ is odd then $\LL_f((2k-1)f+q^{+})=q^{+}$;
\item[(4)] 
$\LL_f(d)=k(f+1+q_{-})+q^{+}+1-j\cdot (k-1)$, if $d=(2k-1)f+q^{+}+1+j$ for $0\leq j\leq q_{-}+1$;
\item[] with the correction that, if $f\leq 4$ or $k>1$ then $k$ 
is to be subtracted from $\LL_f(2kf-1)$;
\item[(5)] $\LL_f(d)=(k+1)(f-1)+3$ if $2kf+3\leq d\leq (2k+1)f-1$;
\item[] with the correction for $k=1$ that 1 is to be added to $\LL_f(d)$, and
\item[] with the correction (for $k=2$) that $\LL_f(4f+3)=3f+1$ instead of $3f$.
\end{itemize}
To this general definition of $\LL_f$ the following
modifications are made for small values of the parameters:
\begin{itemize}
\item for $f=1$:
\begin{itemize}
\item[] $\LL_1(3)=4$;
\item[] add $\frac{d+7}{2}$ to $\LL_1(d)$ for $d\equiv 1\bmod 2$ with $d\geq 7$;
\end{itemize}
\item for $f=2$:
\begin{itemize}
\item[] $\LL_2(2)=1$;
\item[] $\LL_2(6)=2$;
\item[] add $\frac{d+5}{4}$ to $\LL_2(d)$ for $d\equiv 3\bmod 4$;
\item[] add 1 to $\LL_2(d)$ for $d\equiv 0\bmod 4$.
\end{itemize}
\item for $f=3$:
\begin{itemize}
\item[] $\LL_3(d)=2$ for $d\equiv 5\bmod 6$;
\item[] $\LL_3(6)=4$;
\item[] $\LL_3(7)=2$;
\item[] $\LL_3(8)=4$;
\item[] $\LL_3(9)=4$;
\item[] $\LL_3(10)=5$;
\item[] $\LL_3(15)=26$;
\item[] add $\frac{d+20}{6}$ to $\LL_3(d)$ for $d\equiv 4\bmod 6$ 
with $d\geq 16$;
\item[] add 4 to $\LL_3(d)$ for $d\equiv 3\bmod 6$ 
with $d\geq 21$.
\end{itemize}
\item for $f=4$:
\begin{itemize}
\item[] $\LL_4(5)=2$;
\item[] $\LL_4(36)=5$;
\item[] add $\frac{d+44}{8}$ to $\LL_4(d)$ for $d\equiv 4\bmod 8$ 
with $d\geq 44$.
\end{itemize}
\end{itemize}
\end{definition}

\begin{example}\rm
Consider, for example, the case $f=19$, so $q_{-1}=9$ and $q^{+}=10$.
We will here (and later on) use the notation $f[a\cdots b]=[x_a\cdots x_b]$
as shorthand for $f(k)=x_k$ for integers $a\leq k\leq b$, and
$f[a\cdots b]=x$ if $f(k)=x$ for $a\leq k\leq b$.

Then
\begin{itemize}
\item[(1)] $\LL_{19}[1\cdots 18]=19$,
\end{itemize}
and, for $k=1$ we get
\begin{itemize}
\item[(2)] $\LL_{19}[19\cdots 28]=[20, 21, 22, 23, 24, 25, 26, 27, 28, 29]$;
\item[(3)] $\LL_{19}(29)=10$;
\item[(4)] $\LL_{19}[30\cdots 40]=40$;
\item[(5)] $\LL_{19}[41\cdots 56]=39+1$ (by the correction for $k=1$);
\end{itemize}
and the next block of length 38 is ($k=2$):
\begin{itemize}
\item[(2)] $\LL_{19}[57\cdots 66]=[40, 42, 44, \cdots, 58]$;
\item[(3)] $\LL_{19}(67)=10$;
\item[(4)] $\LL_{19}[68\cdots 78]=[69, 68, 67, \cdots, 59]$, with correction
$\LL_{19}(75)=62-2=60$;
\item[(5)] $\LL_{19}[79\cdots 94]=57$, except that
$\LL_{19}(79)=57+1=58$.
\end{itemize}
This gives the following sequence for the first 94 values:
\begin{tiny}
\begin{eqnarray*}
&[&19,19,19,19,19,19,19,19,19,19,19,19,19,19,19,19,19,19,20,\\
&\ &21,22,23,24,25,26,27,28,29,10,40,40,40,40,40,40,40,40,40,\\
&\ &40,40,40,40,40,40,40,40,40,40,40,40,40,40,40,40,40,40,40,\\
&\ &42,44,46,48,50,52,54,56,58,10,69,68,67,66,65,64,63,60,61,\\
&\ &60,59,58,57,57,57,57,57,57,57,57,57,57,57,57,57,57,57\ ].
\end{eqnarray*}
\end{tiny}
The next two blocks of $38$ values are 
\begin{tiny}
\begin{eqnarray*}
&[&60,63,66,69,72,75,78,81,84,87,10,98,96,94,92,90,88,86,81,\\
&\ &82,80,78,75,75,75,75,75,75,75,75,75,75,75,75,75,75,75,75\ ], 
\end{eqnarray*}
\end{tiny}
and
\begin{tiny}
\begin{eqnarray*}
&[&80, 84, 88, 92, 96, 100, 104, 108, 112, 116, 10, 127, 124, 121, 118, 115, 112,109,\\
&\ &102, 103, 100, 97, 93, 93, 93, 93, 93, 93, 93, 93, 93, 93, 93, 93, 93, 93, 
93, 93\ ],
\end{eqnarray*}
\end{tiny}
etcetera.
\end{example}
We are now ready to state the conjecture for
the period length of $S_{f,g}$ for every possible pair
of starting values $f, g$ with $g>f>0$; we switch to the use of $f, d=g-f$
as before and write $\PP(d,f)$ for the length of the period of $S_{f, f+d}$.
\begin{conjecture}\label{conj:per}
For all $f, d=g-f\geq 1$ the strict, greedy 2-sum sequence
$S_{f,g}$ is ultimately periodic with length of the period $\PP(d,f)=\LL_f(d)$.
\end{conjecture}

\begin{example}\rm
Here are the first few conjectured values of $\PP(d,f)$ for $f=1,\ldots,6$:
\begin{tiny}
\begin{eqnarray*}
\PP[1\cdots25,1]&=&[ 1, 1, 4, 1, 3, 1, 11, 1, 13, 1, 15, 1, 17, 1, 19, 1, 21, 1, 23, 1, 25, 1, 27, 1, 29 ],\\
\PP[1\cdots25,2]&=&[ 2, 1, 5, 7, 6, 2, 9, 11, 9, 9, 13, 15, 12, 12, 17, 19, 15, 15, 21, 23, 18, 18, 25, 27, 21 ],\\
\PP[1\cdots25,3]&=&[ 3, 3, 4, 5, 2, 4, 2, 4, 4, 5, 2, 13, 12, 11, 26, 21, 2, 18, 16, 14, 20, 27, 2, 23, 20 ],\\
\PP[1\cdots25,4]&=&[ 4, 4, 4, 5, 2, 7, 10, 10, 10, 10, 10, 10, 12, 14, 15, 16, 15, 14, 13, 15, 18, 21, 21, 22, 20 ],\\
\PP[1\cdots25,5]&=&[ 5, 5, 5, 5, 6, 7, 8, 3, 12, 12, 12, 12, 12, 12, 12, 14, 16, 3, 18, 19, 18, 17, 16, 15, 18 ],\\
\PP[1\cdots 25,6]&=& [ 6, 6, 6, 6, 6, 7, 8, 9, 10, 14, 14, 14, 14, 14, 14, 14, 14, 14, 16, 18, 20, 24, 21, 22, 21 ].
\end{eqnarray*}
\end{tiny}
Compare this with the first columns of computed values
in the second matrix in the Introduction.
\end{example}

\begin{example}\rm
Next consider  the case $f=12$;  then $q_{-}=q^{+}=6$.
\begin{itemize}
\item[(1)] $\LL_{12}[1\cdots 11]=12$;
\end{itemize}
after which (with Step (3) omitted since $f$ is even):
\begin{itemize}
\item[(2)] $\LL_{12}[12\cdots 18]=[13,14,15,16,17,18,19]$;
\item[(4)] $\LL_{12}[19\cdots 26]=26$;
\item[(5)] $\LL_{12}[27\cdots 35]=2(11)+3+1=26$;
\end{itemize}
and then for $k=2$:
\begin{itemize}
\item[(2)] $\LL_{12}[36\cdots 42]=[26,28,30,32,34,36,38]$;
\item[(4)] $\LL_{12}[43\cdots 50]=[45,44,43,42,41,40,39,38]$, with correction
$\LL_{12}[47]=41-2=39$;
\item[(5)] $\LL_{12}[51\cdots 59]=36$, except that
$\LL_{12}[51]=37$.
\end{itemize}
This gives, for the first 59 entries:
\begin{tiny}
\begin{eqnarray*}
&[&12, 12, 12, 12, 12, 12, 12, 12, 12, 12, 12, 13, 14, 15, 16, 17, 18, 19, 26, 26,\\
&& 26, 26, 26, 26, 26, 26, 26, 26, 26, 26, 26, 26, 26, 26, 26, 26, 28, 30, 32, 34,\\
&& 36, 38, 45, 44, 43, 42, 39, 40, 39, 38, 37, 36, 36, 36, 36, 36, 36, 36, 36\ \  ]
\end{eqnarray*}
\end{tiny}
and the next blocks of $24$ values are 
\begin{tiny}
$$[ 39, 42, 45, 48, 51, 54, 57, 64, 62, 60, 58, 53, 54, 52, 50, 47, 47, 47, 47, 
47, 47, 47, 47, 47 ]$$
\end{tiny}
and
\begin{tiny}
$$[ 52, 56, 60, 64, 68, 72, 76, 83, 80, 77, 74, 67, 68, 65, 62, 58, 58, 58, 58, 
58, 58, 58, 58, 58 ],$$
\end{tiny}
etcetera.

Keep in mind that what we are stating here are conjectured
period lengths for 2-sumfree sequences. As a concrete example 
covered in this case, consider $f=12$ and $d=18$,
so $S_{12, 30}$. It turns out that this sequence is indeed periodic with period
length 19 (as predicted above), and preperiod $k=36$. More precisely,
$$D_{12, 30}=(18, 1^{11}, 13, 1^{6}, 22, 1, 19, 1^6, 17, 1^6, 19,
\overline{1^4, 19, 1^5, 18, 1^6, 18, 1}).$$
A short initial segment of $S_{12,30}$ is
\begin{tiny}
$$(12, 30, 31, 32, 33, 34, 35, 36, 37, 38, 39, 40, 41, 54, 55, 56, 57, 58, 59, 
60, 82,$$
$$ 83, 102, 103, 104, 105, 106, 107, 108, 125, 126, 127, 128, 129, 130, 
131)$$
\end{tiny}
which is in fact the preperiod. The first actual period follows:
\begin{tiny}
$$(150, 151, 152, 153, 154, 173, 174, 175, 176, 177, 178, 196, 197, 198, 199, 200, 
201, 202, 220).$$
\end{tiny}
\end{example}
\section{The conjectures for preperiod lengths}\label{sec:preperiod}
Less importantly, and slightly more difficult to state, we
supplement in this section the periodicity conjectures by
precise conjectures on the {\it lengths of preperiods} for $2$-sumfree
sequences. Once both preperiod length $k$ and period length $p$ are known,
from an initial segment of length $k+p$, any term of $S_{f,g}$ can
(in principle, and conjecturally) be obtained in a greedy fashion
(from the difference sequence, or the modulus $m$).

As pointed out in the Introduction, the (pre)period for $S_{f,g}$
is completely determined by that of the sequences of differences
$D_{f,g}$ and that of the characteristic sequence $C_{f,g}$. In this
section we prefer to consider $D_{f,g}$ primarily. 

The general behaviour of preperiod length $\QQ$
for fixed $f$ as a function of $d=g-f$ is as follows.
The preperiod length grows, and the growth rate is most easily
described by distinguishing the residue classes modulo
$2f$ for $d$, again; in most residue classes (namely $f+3$ of them)
the growth is {\it linear} in $d$, but in $f-3$ residue classes
the growth is {\it quadratic} in $d$. 
The version of the conjectures below 
spells out precisely for each of $2f$ residue classes for $d$
what the linear or quadratic growth function will be.
Note that for small values of the parameters
a correction is necessary, which is listed at the end of the definition.

\begin{definition}\label{def:KK}
Define, for the residue classes $r\equiv d\bmod 2f$, 
fixing the additional notation
$d=r+h\cdot (2f)$, and choosing representatives $1\leq r\leq 2f$,
a function $\KK_{r}$ of $h$ and $f$.

\begin{itemize}
\item[] For $f\geq 5$
the function is 
\begin{itemize}
\item[] linear in $h$ for fixed $f$ for the following $f+3$ residue classes $r\bmod 2f$:
\begin{itemize}
\item[(a)] $\KK_{1}(h,f)=(2f+1)\cdot h+(2f+7)$;
\item[(b)] $\KK_{2}(h,f)=(2f)\cdot h+(2f+7)$;
\item[(c)] $\KK_{r}(h,f)=(f+1)\cdot h+(3f+1+r)$
for $r=3, 4,\ldots,f-1$;
\item[(d)] $\KK_{f}(h,f)=(4f+1)\cdot h+(5f-3)$;
\item[(e)] $\KK_{f+1}(h,f)=(2f+3)\cdot h+(2f+7)$;
\item[(f)] $\KK_{2f-1}(h,f)=f\cdot h+(2f+1)$;
\item[(g)] $\KK_{2f}(h,f)=f\cdot h+(f+2)$.
\end{itemize}
\item[] and quadratic in $h$ for fixed $f$ for the remaining $f-3$ cases, as given below,
using the notation $q=\lfloor\frac{f+1}{2}\rfloor$ and $s=r-(f+q)$:
\begin{itemize}
\item[(h)] for even $f\geq 6$, and $f+3\leq r\leq 2f-2$:
\begin{itemize}
\item[(h$_0$)] $\KK_{f+q}(h,f)=\frac{9f+14}{4}h^2+\frac{15f+42}{4}h+
\frac{3f+2}{2}$;
\item[(h$_1$)] $\KK_{r}(h,f)=\frac{9f+16-2s}{4}h^2+\frac{15f+40+10s}{4}h+\frac{3f+2s}{2}$ when $s>0$; 
\item[(h$_2$)] $\KK_{r}(h,f)=\frac{9f+16+2s}{4}h^2+
\frac{15f+42+6s}{4}h+\frac{3f+2+2s}{2}$
when $s<0$; 
\end{itemize}
\item[(i)] while for odd $f\geq 5$, and $f+3\leq r\leq 2f-2$:
\begin{itemize}
\item[(i$_0$)] $\KK_{f+q}(h,f)=\frac{9f+15}{4}h^2+\frac{15f+45}{4}h+
\frac{3f-1}{2}$;
\item[(i$_1$)] $\KK_{r}(h,f)=\frac{9f+15-2s}{4}h^2+
\frac{15f+45+10s}{4}h+\frac{3f+1+2s}{2}$
when $s>0$; 
\item[(i$_2$)] $\KK_{r}(h,f)=\frac{9f+15+2s}{4}h^2+
\frac{15f+45+6s}{4}h+\frac{3f+3+2s}{2}$
when $s<0$; 
\end{itemize}
\item[(j)] when $r=f+2$ alternatingly, depending on $d\bmod 4f$:
\begin{itemize}
\item[(j$_0$)] $\KK_{r}(h,f)=(4f+9)h^2+(4f+9)h$
when $d\equiv f+2\bmod 4f$;
\item[(j$_1$)] $\KK_{r}(h,f)=(4f+9)h^2+(4f+9)h-1$
when $d\equiv 2f+f+2\bmod 4f$.
\end{itemize}
\item[(k)] Finally, for small values of $d$ (that is, of $h$ and $r$),
the following corrections are to be made to $\KK_{r}$ for
$f\geq 5$:
\begin{itemize}
\item[(k$_0$)] if $h=0$, subtract $2f+r$ from $\KK_{r}$;
\item[(k$_1$)] if $h=1$, subtract $3f+r$ from $\KK_{r}$;
\item[(k$_2$)] if $r=1$ and $h=0$, add $f-5$ to $\KK_{1}$;
\item[(k$_3$)] if $r=1$ and $h=1$, subtract $5$ from $\KK_{1}$;
\item[(k$_4$)] if $r=1$ and $h=2$ subtract $4f+7$ from $\KK_{1}$;
\item[(k$_5$)] if $r=2$ and $h=0$ add $f-4$ to $\KK_{2}$;
\item[(k$_6$)] if $r=2$ and $h=1$, subtract $3$ from $\KK_{2}$;
\item[(k$_7$)] if $r=3$ and $h=2$, add $4f+1$ to $\KK_{3}$.
\end{itemize}
\end{itemize}
\end{itemize}
\item[] For $1\leq f\leq 4$ we define
\begin{itemize}
\item[$\bullet$] for $f=1$:
\begin{itemize}
\item[] $$\KK_{1}(1+2h,1)=
\begin{cases}
2,&\text{if\ } h=0;\\
4,&\text{if\ } h=1;\\
5,&\text{if\ } h=2;\\
h+4,&\text{if\ } h\geq 3.
\end{cases}$$
\item[] $$\KK_{2}(2+2h,1)=
\begin{cases}
0,&\text{if\ } h=0;\\
1,&\text{if\ } h\geq 1.
\end{cases}$$
\end{itemize}
\item[$\bullet$] for $f=2$:
\begin{itemize}
\item[] $$\KK_{1}(1+4h,2)=
\begin{cases}
3,&\text{if\ } h=0;\\
4,&\text{if\ } h=1;\\
6,&\text{if\ } h=2;\\
5h+11,&\text{if\ } h\geq 3.
\end{cases}$$
\item[] $$\KK_{2}(2+4h,2)=
\begin{cases}
2,&\text{if\ } h=0;\\
5,&\text{if\ } h=1;\\
7,&\text{if\ } h=2;\\
5h+11,&\text{if\ } h\geq 0.
\end{cases}$$
\item[] $$\KK_{3}(3+4h,2)=2h+5.$$
\item[] $$\KK_{4}(4+4h,2)=2h+6.$$
\end{itemize}
\item[$\bullet$] for $f=3$:
\begin{itemize}
\item[] $$\KK_{1}(1+6h,3)=
\begin{cases}
4,&\text{if\ } h=0;\\
5,&\text{if\ } h=1;\\
8,&\text{if\ } h=2;\\
7h+13,&\text{if\ } h\geq 3.
\end{cases}$$
\item[] $$\KK_{2}(2+6h,3)=
\begin{cases}
4,&\text{if\ } h=0;\\
5,&\text{if\ } h=1;\\
6h+13,&\text{if\ } h\geq 2.
\end{cases}$$
\item[] $$\KK_{3}(3+6h,3)=
\begin{cases}
4,&\text{if\ } h=0;\\
19,&\text{if\ } h=1;\\
36,&\text{if\ } h=2;\\
81,&\text{if\ } h=3;\\
10h+19,&\text{if\ } h\geq 4.
\end{cases}$$
\item[] $$\KK_{4}(4+6h,3)=
\begin{cases}
4,&\text{if\ } h=0;\\
9h+14,&\text{if\ } h\geq 1.
\end{cases}$$
\item[] $$\KK_{5}(5+6h,3)=
\begin{cases}
9,&\text{if\ } h=0;\\
26,&\text{if\ } h=1;\\
\frac{21}{2}h^2+\frac{27}{2}h,&\text{if\ } h\geq 2.
\end{cases}$$
\item[] $$\KK_{6}(6+6h,3)=3h+5.$$
\end{itemize}
\item[$\bullet$] for $f=4$:
\begin{itemize}
\item[] $$\KK_{1}(1+8h,4)=
\begin{cases}
5,&\text{if\ } h=0;\\
6,&\text{if\ } h=1;\\
10,&\text{if\ } h=2;\\
9h+15,&\text{if\ } h\geq 3.
\end{cases}$$
\item[] $$\KK_{2}(2+8h,4)=
\begin{cases}
5,&\text{if\ } h=0;\\
6,&\text{if\ } h=1;\\
8h+15,&\text{if\ } h\geq 2.
\end{cases}$$
\item[] $$\KK_{3}(3+8h,4)=
\begin{cases}
5,&\text{if\ } h=0;\\
6,&\text{if\ } h=1;\\
43,&\text{if\ } h=2;\\
5h+16,&\text{if\ } h\geq 3.
\end{cases}$$
\item[] $$\KK_{4}(4+8h,4)=
\begin{cases}
5,&\text{if\ } h=0;\\
12,&\text{if\ } h=1;\\
116,&\text{if\ } h=2;\\
112,&\text{if\ } h=3;\\
193,&\text{if\ } h=4;\\
32h+40,&\text{if\ } h\geq 5.
\end{cases}$$
\item[] $$\KK_{5}(5+8h,4)=
\begin{cases}
5,&\text{if\ } h=0;\\
13,&\text{if\ } h=1;\\
11h+15,&\text{if\ } h\geq 2.
\end{cases}$$
\item[] $$\KK_{6}(6+8h,4)=
\begin{cases}
5,&\text{if\ } h=0;\\
14,&\text{if\ } h=1;\\
20,&\text{if\ } h=2;\\
51,&\text{if\ } h=3;\\
\frac{25}{4}h^2-7,&\text{if\ } h\geq 4,\text{\ even},\\
\frac{25}{4}h^2-\frac{29}{4}&\text{if\ } h\geq 5\text{\ odd}.
\end{cases}$$
\item[] $$\KK_{7}(7+8h,4)=
\begin{cases}
11,&\text{if\ } h=0;\\
4h+9,&\text{if\ } h\geq 1.
\end{cases}$$
\item[] $$\KK_{8}(8+8h,4)=4h+6.$$
\end{itemize}
\end{itemize}
\end{itemize}
\end{definition}

The main conjecture about the preperiod length is now easily stated,
when we use $\QQ(d,f)$ for the length of the preperiod,
for each pair of positive integers $f, d=g-f$:
\begin{conjecture}\label{conj:pre}
For all $f, d=g-f\geq 1$ the strict, greedy 2-sum sequence
$S_{f,g}$ is ultimately periodic with preperiod length
	$\QQ(d,f)=\KK_{r}(d,f)$, where $r\equiv d\bmod 2f$.
\end{conjecture}

\begin{example}\rm
Here are the first few conjectured values of $\QQ(d,f)$ for $f=1,\ldots,6$:
\begin{tiny}
\begin{eqnarray*}
\QQ[1\cdots25,1]&=&[ 
2, 0, 4, 1, 5, 1, 7, 1, 8, 1, 9, 1, 10, 1, 11, 1, 12, 1, 13, 1, 14, 1, 15, 1, 16
],\\
\QQ[1\cdots25,2]&=&[ 
3, 2, 5, 6, 4, 5, 7, 8, 6, 7, 9, 10, 26, 26, 11, 12, 31, 31, 13, 14, 36, 36, 15, 16, 41
],\\
\QQ[1\cdots25,3]&=&[ 
4, 4, 4, 4, 9, 5, 5, 5, 19, 23, 26, 8, 8, 25, 36, 32, 69, 11, 34, 31, 81, 41, 135, 14, 41
],\\
\QQ[1\cdots25,4]&=&[ 
5, 5, 5, 5, 5, 5, 11, 6, 6, 6, 6, 12, 13, 14, 13, 10, 10, 31, 43, 116, 37, 20, 17, 14, 42
],\\
\QQ[1\cdots25,5]&=&[ 
6, 6, 6, 6, 6, 6, 6, 30, 13, 7, 7, 7, 7, 7, 14, 15, 16, 71, 16, 12, 12, 37, 52, 32, 54
],\\
\QQ[1\cdots 25,6]&=&[
7, 7, 7, 7, 7, 7, 7, 7, 21, 25, 15, 8, 8, 8, 8, 8, 8, 16, 17, 18, 58, 63, 19, 14, 14
].
\end{eqnarray*}
\end{tiny}

Compare this with the first columns of computed values in
the first matrix in the Introduction.
\end{example}

\begin{example}\rm
Let us describe the case $f=9$ in detail. To show that there is no
obvious pattern at first sight, we simply list the length of preperiods for
$S_{9, 9+d}$ that we computed for $1\leq d\leq 234$:

\begin{tiny}
\begin{eqnarray*}
&&10, 10, 10, 10, 10, 10, 10, 10, 10, 10, 10, 27, 28, 48, 35, 37, 21, 11, 11, 11, 11, 11, 11, 11, 11, 11, 22, 23, \\
&&24, 76, 79, 115, 88, 90, 28, 20, 20, 61, 88, 52, 53, 54, 55, 56, 90, 67, 35, 148, 153, 203, 164, 166, 37, 29, 82, \\
&&79, 61, 62, 63, 64, 65, 66, 117, 88, 91, 289, 301, 364, 312, 315, 46, 38, 101, 97, 71, 72, 73, 74, 75, 76, 192, 109,\\
&&168, 478, 495, 577, 510, 512, 55, 47, 120, 115, 81, 82, 83, 84, 85, 86, 229, 130, 269, 713, 736, 838, 755, 755, 64, \\
&&56, 139, 133, 91, 92, 93, 94, 95, 96, 266, 151, 393, 994, 1024, 1147, 1047, 1044, 73, 65, 158, 151, 101, 102, 103, \\
&&104, 105, 106, 303, 172, 539, 1321, 1359, 1504, 1386, 1379, 82, 74, 177, 169, 111, 112, 113, 114, 115, 116, 340, \\
&&193, 708, 1694, 1741, 1909, 1772, 1760, 91, 83, 196, 187, 121, 122, 123, 124, 125, 126, 377, 214, 899, 2113, 2170, \\
&&2362, 2205, 2187, 100, 92, 215, 205, 131, 132, 133, 134, 135, 136, 414, 235, 1113, 2578, 2646, 2863, 2685, 2660, \\
&&109, 101, 234, 223, 141, 142, 143, 144, 145, 146, 451, 256, 1349, 3089, 3169, 3412, 3212, 3179, 118, 110, 253, \\ 
&&241, 151, 152, 153, 154, 155, 156, 488, 277, 1608, 3646, 3739, 4009, 3786, 3744, 127, 119
\end{eqnarray*}
\end{tiny}

Next, to discern the pattern suggested by Conjecture \ref{conj:pre},
we list them by residue class for $d\bmod 18$:

\begin{tiny}
\begin{eqnarray*}
&d\equiv 1\bmod 18: &10, 11, 20, 82, 101, 120, 139, 158, 177, 196, 215, 234, 253,\cdots \hskip.5truecm (19)\\
&d\equiv 2\bmod 18: &10, 11, 61, 79, 97, 115, 133, 151, 169, 187, 205, 223, 241,\cdots \hskip.5truecm (18)\\
&d\equiv 3\bmod 18: &10, 11, 88, 61, 71, 81, 91, 101, 111, 121, 131, 141, 151,\cdots \hskip.5truecm (10)\\
&d\equiv 4\bmod 18: &10, 11, 52, 62, 72, 82, 92, 102, 112, 122, 132, 142, 152,\cdots \hskip.5truecm (10)\\
&d\equiv 5\bmod 18: &10, 11, 53, 63, 73, 83, 93, 103, 113, 123, 133, 143, 153,\cdots \hskip.5truecm (10)\\
&d\equiv 6\bmod 18: &10, 11, 54, 64, 74, 84, 94, 104, 114, 124, 134, 144, 154,\cdots \hskip.5truecm (10)\\
&d\equiv 7\bmod 18: &10, 11, 55, 65, 75, 85, 95, 105, 115, 125, 135, 145, 155,\cdots \hskip.5truecm (10)\\
&d\equiv 8\bmod 18: &10, 11, 56, 66, 76, 86, 96, 106, 116, 126, 136, 146, 156,\cdots \hskip.5truecm (10)\\
&d\equiv 9\bmod 18: &10, 22, 90, 117, 192, 229, 266, 303, 340, 377, 414, 451, 488,\cdots \hskip.5truecm (37)\\
&d\equiv 10\bmod 18: &10, 23, 67, 88, 109, 130, 151, 172, 193, 214, 235, 256, 277,\cdots \hskip.5truecm (21)\\
&d\equiv 11\bmod 18: &10, 24, 35, 91, 168, 269, 393, 539, 708, 899, 1113, 1349, 1608,\cdots \hskip.5truecm ((23/22)) \\
&d\equiv 12\bmod 18: &27, 76, 148, 289, 478, 713, 994, 1321, 1694, 2113, 2578, 3089, 3646,\cdots \hskip.5truecm ((46))\\
&d\equiv 13\bmod 18: &28, 79, 153, 301, 495, 736, 1024, 1359, 1741, 2170, 2646, 3169, 3739,\cdots  \hskip.5truecm ((47)) \\
&d\equiv 14\bmod 18: &48, 115, 203, 364, 577, 838, 1147, 1504, 1909, 2362, 2863, 3412, 4009,\cdots   \hskip.5truecm ((48))\\
&d\equiv 15\bmod 18: &35, 88, 164, 312, 510, 755, 1047, 1386, 1772, 2205, 2685, 3212, 3786,\cdots  \hskip.5truecm ((47)) \\
&d\equiv 16\bmod 18: &37, 90, 166, 315, 512, 755, 1044, 1379, 1760, 2187, 2660, 3179, 3744,\cdots   \hskip.5truecm ((46))\\
&d\equiv 17\bmod 18: &21, 28, 37, 46, 55, 64, 73, 82, 91, 100, 109, 118, 127,\cdots \hskip.5truecm (9)\\
&d\equiv 18\bmod 18: &11, 20, 29, 38, 47, 56, 65, 74, 83, 92, 101, 110, 119,\cdots  \hskip.5truecm (9)
\end{eqnarray*}
\end{tiny}

\noindent
where at the end of 12 lines we indicated the common differences for
that line in parentheses, and at the end of 6 lines the common
second differences for that line in double parentheses
(often exhibited only after the first few terms). This suggests the
linear or quadratic growth rate, expressed in Definition \ref{def:KK}.

For $d\equiv 11\bmod 18$ we consider the two subsequences for
$d\equiv 11\bmod 36$ and $d\equiv 29\bmod 36$ of $11\bmod 18$ separately, 
and will find regular (but slightly different)  quadratic growth for both,
as expressed in case (j) of Definition \ref{def:KK}.
\end{example}

\section{Confirmed cases}\label{sec:proof}
In this section we give three results. The first one,
Theorem \ref{thm:upper}, appeared before
in \cite{BeBo}, and shows that the obvious pattern alluded to in
the Introduction for the values of (pre)period lengths
`above the diagonal', i.e., for $d<f$, is essentially correct.
The second result, Theorem \ref{thm:diag},
is new and extends this to the `diagonal' case $d=f$.
The third one, Theorem \ref{corollary}, shows how a finite
computation suffices to {\it prove} periodicity 
(of 2-sumfree sequences).
We offer proofs for all results in this section.

\begin{theorem}[\cite{BeBo}, Thm 13]\label{thm:upper}
Let $f\geq 1$. 
For every $g\leq 2f-1$ the 2-sumfree sequence $S_{f,g}$ is
characterized as follows:
\begin{eqnarray*}
z\in S_{f,g}&\iff& z=f \textrm{\ \ or\ \ } z=g\textrm{\ \ or\ \ }g+1\leq z\leq f+g-1\textrm{\ \ or\ \ }\\
&&z=w+k\cdot (2f+g-1), \textrm{\ \ for\ \ } k\geq 0, \textrm{\ with}\\
&&\qquad w\in\{ 2f+2g-2, 2f+2g-1, \cdots 3f+2g-3\};
\end{eqnarray*}
in particular,  $S_{f,g}$ 
is periodic with period length $f$
and preperiod of length $f+1$.
\end{theorem}

\begin{proof}
Let $T$ be the infinite sequence defined by the right hand side of the
equivalence in the main statement. We will use $d=g-f$ as before.
We first show that all elements of each of the remaining
residue classes modulo $m=3f+d-1$ is the sum of two elements in $T$.
Note that $T$ contains precisely $f$ residue classes modulo $m$, and we
consider the remaining $m-f=2f+d-1$ classes. By $\overline{z}=\overline{w}+y$ we indicate that each element in the residue class
of $z\bmod m$ can be written as a sum of $y$ and an element of $w\bmod m$.
\begin{itemize}
    \item[]
    \begin{itemize}
    \item[$2f+d-1$:] Note that $2f+d-1\in T$, and for $k>0$ both $f+d-1+k\cdot m$ and $f$ are in $T$, hence $2f+d-1+km$ is not;
    \item[$2f+d$\phantom{$-1$}:] $2f+d+k\cdot m=(f+d+k\cdot m)+f\notin T$;
    \item[$\vdots\qquad$]
    \item[$3f+d-2$:] $3f+d-2+k\cdot m=(2f+d-2+k\cdot m)+f\notin T$; 
    \item[$3f+d-1$:] Note that $3f+d-1=(2f+d-1)+f\notin T$, while for $k>0$ we have $3f+d-1+k\cdot m=(k+1)\cdot m=(f+d-1+k\cdot m)+2f\notin T$;
    \item[$3f+d \phantom{-1}$:] $3f+d+k\cdot m=(2f+d-1+k\cdot m)+f+d\notin T$;
    \item[$\vdots\qquad$]
    \item[$f+d-2$:] $f+d-2+k\cdot m=(2f+d-2+k\cdot m)+2f+d-1\notin T$;
\end{itemize}
\end{itemize}
Next we show that none of the elements of $T$ is itself 
the sum of two different elements of $T$. If $z_1$ and $z_2$ are contained in different full residue classes of $T$, then for the smallest
residue holds:
$$2f+d-2< 2f+2d-1\leq z_1+z_2\leq 4f+2d-5<f+d-1+m,$$
and hence the sum of the classes of $z_1$ and $z_2$ is outside $T$.
Similarly, for the sum of $z$ and one of the additional elements $f$ 
and $2f+d-1$, we find 
$$2f+d-1\leq z+f\leq 3f+d-2<m+f+d-1$$
and
$$2f+d-2<3f+2d-2\leq z+2f+d-1\leq 4f+2d-3<m+f+d-1$$
so these sums are not in $T$.
\end{proof}

The next result for the `diagonal' case $g=2f$, so $d=g-f=f$,
did not appear before.
It resembles the previous case (but is not simply obtained by taking $g=2f$
in there); we include it also to show
that it is not difficult to prove results for more families
of 2-sumfree sequences, but the proofs are tedious.
We first give examples.
\begin{example}\rm
In the theorem below, which deals with the sequences $S_{f,2f}$,
the smallest two cases, for $f=1$ and $f=2$, are excluded. It is easy to see
that they commence as follows
$$S_{1,2}=(1, 2, 4, 7, 10, 13, 16, 19, 22, \cdots)$$
and
$$S_{2,4}=(2, 4, 5, 8, 11, 14, 17, 20, 23, \cdots).$$
The dots here are used to indicate an indefinite
continuation with periodic behaviour
by steps of 3; it is not hard to prove that this is correct.

The first case to which the theorem applies is $f=3$, where one finds
$$S_{3,6}=(3, 6, 7, 8, 12, 16, 17, 21, 26, 30, 31, 35, 40, 44, 45, 49, 54, \cdots)$$
for which it is slightly harder to see the continuation: after the
fourth term (the preperiod) the period (of length 4) starts, with differences
$4, 1, 4, 5$. This continues forever.
\end{example}
\begin{theorem}\label{thm:diag}
For every $f\geq 3$ the 2-sumfree sequence $S_{f,2f}$ is
characterized as follows:
\begin{eqnarray*}
z\in S_{f,2f}&\iff& z=f \textrm{\ \ or\ \ } z=2f\textrm{\ \ or\ \ }2f+1\leq z\leq 3f-1\textrm{\ \ or\ \ }\\
&&z=w+k\cdot (5f-1), \textrm{\ \ for\ \ } k\geq 0, \textrm{\ with}\\
&&\qquad w\in\{ 4f, 6f-2, 6f-1, 7f, 7f+1,\cdots 8f-3\};
\end{eqnarray*}
in particular,  $S_{f,2f}$ 
is periodic with period length $f+1$
and preperiod length $f+1$.
\end{theorem} 

\begin{proof}
First we show that up to $4f$ the only entries of $S=S_{f,2f}$
will be $f, 2f$ and $2f+1, \cdots, 3f-1$. By definition, $f$ and
$2f$ are the first two elements of $S$, and it will be clear that
then their sum $3f$ is the first non-entry, and each of
$2f+1, \cdots, 3f-1$ is in $S$. But adding $f$ to each of these, shows
that none of $3f+1, \cdots, 4f-1$ will be in $S$.
So up to $4f$ the elements of $S$ are
$$f, 2f, 2f+1, \cdots, 3f-1.$$

Next we consider the block of $5f-1$ integers from $4f$ up to $9f-1$.
The next smallest sum of elements of $S$ will be $2f+(2f+1)=4f+1$,
so $4f$ is in $S$; but all integers from $4f+1$ up to
$6f-3=(3f-2)+(3f-1)$, as sums of two elements from the initial
segment, are not in $S$. The next smallest sums are
$f+4f=5f$ and $2f+4f=6f$; but $5f<6f-3$ as $sf>3$, so the new
gap in $S$ we find starts $6f$ and ends with $3f-1+4f=7f-1$.
Up to $7f$ we then have
$$f, 2f, 2f+1, \cdots, 3f-1, 4f, 6f-2, 6f-1,$$
from which we cannot produce $7f$ as distinct sum. The next smallest sum
we can make is $2f+(6f-2)=8f-2$ (using again that $f>2$), so now $S$ starts
out as
$$f, 2f, 2f+1, \cdots, 3f-1, 4f, 6f-2, 6f-1, 7f, \cdots, 8f-3.$$
But also $8f-1=2f+(6f-1), \cdots 9f-2=(3f-1)+(6f-1)$ are all sums,
hence not in $S$, while $9f-1$ is not a sum. 

Now we repeat the procedure for the next block of $5f-1$ integers
from $9f-1$ up to $14f-2$:
$9f=2f+7f$ up to $11f-4=(3f-1)+(8f-3)$ are sums, but $11f-3$ and $11f-2$ are
not. Then $11f-1=2f+(9f-1)$ up to $12f-2=(3f-1)+9f-1)$ are all excluded,
but $12f-1, \cdots, 13f-4$ are in $S$.
Also $13f-2=2f+(11f-2)$ up to $14f-3=(3f-1)+11f-2)$ are 2-sums, so not in $S$,
while $14f-2$ is not a sum. This can then be repeated, for blocks
of $5f-1$ elements, indefinitely.

We have now found three blocks of $f+1$ elements of $S$:
$$\begin{array}{lllllcl}
&      &       &f,    &2f, &2f+1, \cdots, 3f-2, &3f-1\\
&4f,   &6f-2,  &6f-1, &    &7f,   \cdots, 8f-3  &\\
&9f-1, &11f-3, &11f-2,&    &12f-1,\cdots, 13f-4 &
\end{array}$$
and can produce as many blocks of $f+1$ elements as we like,
shifted by $5f-1$, from here on.
Looking at the differences between entries
we see that the period starts properly after the first block,
which forms the preperiod.
\end{proof}
To recognize (with certainty), from an initial segment,
that a sumfree sequence becomes periodic, is complicated
by the possibility that long preperiods may occur. How can
one be sure that an observed repeated pattern will repeat
indefinitely, and is not in fact part of the preperiod?
Lemma 6 of \cite{BeBo}
provides an answer
for general $t$-sumfree sequences. 
The following is the special case  $t=2$,
which we prove here too.

\begin{theorem}\label{corollary}
Let $S$ be a $2$-sumfree sequence.
Suppose that for some $M>0$ the initial segment of length $5M$ of
the characteristic sequence $C$ of $S$
has the property that the 5 consecutive
blocks of length $M$ of which it is composed
are, with the possible exception of the first block, identical and non-null.
Then $C$, and hence $S$, is ultimately periodic.
\end{theorem}
\begin{proof}
Denote the first 5 blocks of length $M$ of $C$ by $M_0, M_1, \cdots, M_4$,
where $C_i\in M_j$ iff $j\cdot M\leq C_i<(j+1)\cdot M$ for $j\in\{0,1,2,3,4\}$.
The assumption is that $M_1=M_2=M_3=M_4\neq 0^M$.

Suppose that $C$ is not ultimately periodic; this will show
by the existence of a minimal index $v\geq 5M$ such that
$C_v\neq C_{v-M}=C_{v-2M}=C_{v-3M}=C_{v-4M}$. Without loss of generality
we assume that $5M\leq v<6M$; if not, there would be more than
4 repeating blocks of length $M$ and the argument is easily adapted.
Note that by the assumption of non-nullity 
and by construction of $S$, it holds that
$C_u=0$ if and only if $u$ is the sum of two elements in $S$.

If $C_v=1$, that is, $v\in S=S_{f,g}$, then $C_{v-M}=0$, so there exist
$q_1<q_2$ in $S$ with $v-M=q_1+q_2$. Clearly $v=q_1+(q_2+M)$
and $q_2+M<v$, so by minimality of $v$ also $q_2+M\in S$. But
then $v$ is the sum of two elements in $S$, hence not in $S$, a contradiction.

If $C_v=0$, so $v\notin S$ but $v-M\in S$,
then there will exist $r_1<r_2$ in $S$ such that
$v=r_1+r_2$. Now $r_2<v$ but $r_2>2M$, and so
also $r_2-M\in S$;
this leads to the contradiction $v-M=r_1+r_2-M$ as a sum of two elements 
of $S$, unless $r_2-M=r_1$. But in that case $r_1\geq 2M$, since 
$v=r_1+r_2\geq 5M$, which implies that $r_1-M\in S$, like $r_1$. 
Thus $v-M=(r_1-M)+r_2\notin S$, also a contradiction.
\end{proof}

\section{Computational evidence}\label{sec:computations}
Before we had any indication about the correctness of Conjecture \ref{conj:main}
or the growth of (pre)period lengths, we simply computed initial segments
of $S_{f,g}$ for all $f, d\leq 250$, using a simple program in the computer
algebra system {\sc Magma}, \cite{magma}. These data were used to confirm
our suspicion that these sequences are usually ultimately periodic, and to
form the first version of Conjecture \ref{conj:per}. If necessary, we
extended the initial segment to be able to use Corollary \ref{corollary}
to obtain this result.
\begin{theorem}
For all $f\leq 250$ and all $g>f$ with  $g\leq f+250$ the sequences
$S_{f,g}$ are ultimately periodic, and the length of their
period and preperiod is correctly predicted by Conjectures \ref{conj:per}
and \ref{conj:pre}.
\end{theorem}
Once we had this confirmation of the period length conjectures, we decided
to extend our computations to all $f\leq 500$ and $d=g-f\leq 500$. However,
at this stage we had the impression that the preperiod lengths were too
erratic to be able to make accurate predictions. Therefore we only
stored the period lengths, not the preperiod lengths,
for the 250000 2-sum sequences.
\begin{theorem}
For all $f\leq 500$ and all $g\leq f+500$ the sequences
$S_{f,g}$ are ultimately periodic, and the length of their
period is correctly predicted by Conjecture \ref{conj:per}.
\end{theorem}
\begin{example}\rm
To give an indication of the effort involved we provide some numbers.

Of the 62500 preperiods initially computed ($f,d\leq 250$),
853 exceeded 1000 and 10 of those exceeded 10000; the largest
is 17341, for $f=3$ and $g=248$. In contrast, the period of
$S_{3, 248}$ has length 2.
Of the 250000 period lengths computed, 5422 exceed 500, the largest
being 666 (which occurs twice, for $S_{332, 831}$ and $S_{332, 832}$).

But $S_{3, 497}$ has preperiod length 71709 (period 2).
\end{example}


\begin{thebibliography}{10}
\bibitem{BeBo}
Daan van Berkel, Wieb Bosma, {\sl On $t$-sumfree sequences}, preprint.
\bibitem{magma}
Wieb Bosma, John Cannon, Catherine Playoust, 
{\sl The Magma algebra system I: The
user language}, J.~Symb.~Comput.~ {\bf 24} (1997), 235-–265.
\end{thebibliography}
\end{document}